\documentclass{amsart}

\usepackage{amsmath}
\usepackage{amsthm}
\usepackage{amsfonts}
\usepackage{amssymb}
\usepackage{enumerate}

\newcommand{\E}{\mathbb{E}}
\newcommand{\Prob}{\mathbb{P}}
\newcommand{\C}{\mathbb{C}}
\newcommand{\R}{\mathbb{R}}

\newcommand{\T}{^{\mbox{\begin{tiny}T\end{tiny}}}}
\newcommand{\ud}{\mathrm{d}}

\newcommand{\var}{\mathrm{Var}}

\newcommand{\tr}{\mathrm{Tr}}
\renewcommand\Re{\operatorname{Re}}
\renewcommand\Im{\operatorname{Im}}

\theoremstyle{plain}
  \newtheorem{theorem}{Theorem}

  \newtheorem{lemma}[theorem]{Lemma}

\theoremstyle{definition}
  \newtheorem{definition}[theorem]{Definition}
  
  \newtheorem{remark}[theorem]{Remark}

\newcommand{\entry}[5][G]{\ensuremath{#1_{#2,#3;#4,#5}}}

\begin{document}
\title{Products of Independent Non-Hermitian Random Matrices}
\author[S. O'Rourke]{Sean O'Rourke} \thanks{S. O'Rourke has been supported in part by NSF grants VIGRE DMS-0636297 and DMS-1007558}
\address{Department of Mathematics, University of California, Davis, One Shields Avenue, Davis, CA 95616-8633  }
\email{sdorourk@math.ucdavis.edu}

\author[A. Soshnikov]{Alexander Soshnikov}
\address{Department of Mathematics, University of California, Davis, One Shields Avenue, Davis, CA 95616-8633  }
\thanks{A.Soshnikov has been supported in part by NSF grant DMS-1007558}
\email{soshniko@math.ucdavis.edu}

\begin{abstract}
For fixed $m>1$, we consider $m$ independent $n \times n$ non-Hermitian random matrices $X_1, \ldots, X_m$ with i.i.d. centered entries with a finite \\
$(2+\eta)$-th moment, $\ \eta>0.\ $As $n$ tends to infinity, 
we show that the empirical spectral distribution of $n^{-m/2}\* X_1 X_2 \cdots X_m$ converges, with probability $1$, to a non-random, 
rotationally invariant distribution with compact support in the complex plane.  The limiting distribution is the $m$-th power of the circular law.   
\end{abstract}

\maketitle

\section{Introduction and Formulation of Results}

Many important results in random matrix theory pertain to Hermitian random matrices.  Two powerful tools used in this area are the moment method and the 
Stieltjes transform.  Unfortunately, these two techniques are not suitable for dealing with non-Hermitian random matrices, \cite{bai-book}.  

\subsection{The Circular Law}

One of the fundamental results in the study of non-Hermitian random matrices is the circular law.  We begin by defining the empirical spectral 
distribution (ESD).  

\begin{definition}
	Let $X$ be a matrix of order $n$ and let $\lambda_1, \ldots, \lambda_n$ be the eigenvalues of $X$.  Then the \emph{empirical spectral 
distribution} (ESD) $\mu_{X}$ of $X$ is defined as
\begin{equation*}
	\mu_{X}(z,\bar{z}) = \frac{1}{n} \# \left\{k\leq n : \Re\left(\lambda_k\right) \leq \Re(z); \Im\left(\lambda_k\right) \leq \Im(z) \right\}.
\end{equation*}
\end{definition}

Let $\xi$ be a complex random variable with finite non-zero variance $\sigma^2$ and let $N_n$ be a random matrix of order $n$ with entires being i.i.d. 
copies of $\xi$.  We say that the circular law holds for $\xi$ if, with probability $1$, the 
ESD $\mu_{\frac{1}{\sigma \sqrt{n}} N_n}$ of $\frac{1}{\sigma \sqrt{n}} N_n$ converges (uniformly) to the uniform distribution over the unit disk as $n$ 
tends to infinity.  

The circular law was conjectured in the 1950's as a non-Hermitian counterpart to Wigner's semi-circle law.  The circular law was first shown by 
Mehta in 1967 \cite{meh} when $\xi$ is complex Gaussian.  Mehta relied upon the joint density of the eigenvalues which was discovered by Ginibre 
\cite{gin} two years earlier.  

Building on the work of Girko \cite{girko}, Bai proved the circular law under the conditions that $\xi$ has finite sixth moment and that the joint 
distribution of the real and imaginary parts of $\xi$ has bounded density, \cite{bai}.  In \cite{bai-book}, 
the sixth moment assumption was weakened to $\E|\xi|^{2+\eta}$ for any specified $\eta>0$, but the bounded density assumption still remained.  
G\"{o}tze and Tikhomirov (\cite{gotze1}) proved the circular law in the case of i.i.d. sub-Gaussian matrix entries.
Pan and Zhou proved the circular law for any distribution $\xi$ with finite fourth moment \cite{pan} by building on \cite{gotze1} and
utilizing the work of Rudelson and Vershynin in \cite{ver}.  
In an important development, G\"{o}tze and Tikhomirov  showed in \cite{gotze} that the expected spectral distribution 
$\E\mu_{N_n}$ converges to the uniform distribution over the unit disk as 
$n$ tends to infinity assuming that
$ \sup_{jk} \E|(N_n)_{jk}|^2 \phi((N_n)_{jk}) < \infty, $ where $\phi(x)= (\ln(1+|x|))^{19+\eta}, \ \eta>0. $
In \cite{tao}, Tao and Vu proved the circular law assuming a bounded 
$(2+\eta)^{\text{th}}$ moment, for any fixed $\eta>0$.  Finally, Tao and Vu have been able to remove the extra $\eta $ in the moment condition.  Namely,
they proved the circular law in \cite{taovu} assuming only that the second moment is bounded.

\subsection{Main Results}
In this paper, we study the ESD of the product $$ \ X^{(n)} = X^{(n)}_1 X^{(n)}_2 \cdots X^{(n)}_m $$ 
of $m$ independent $n \times n$ non-Hermitian random matrices as $n$ tends to infinity.  
Burda, Janik, and Waclaw \cite{burda} studied the mathematical expectation of the limiting ESD, $\lim_{n \to \infty} \E \mu_X^{(n)}$ ,  
in the case that the entries of the matrices are Gaussian.  
Here we extend their results by proving the almost sure convergence of the ESD, $\mu_X^{(n)},$ for a class of non-Gaussian random matrices.  
Namely, we require that the entries of $X^{(n)}_i, \ i=1,\ldots, m, \ $ are i.i.d. random variables with a finite moment of order $2+\eta, \ \eta>0. $

\begin{theorem} \label{thm:main}
Fix $m>1$ and let $\xi$ be a complex random variable with variance $1$ such that $\Re(\xi)$ and $\Im(\xi)$ are independent each with mean zero and 
$\E|\xi|^{2+\eta} < \infty$ for some $\eta>0$.  Let $X^{(n)}_1, \ldots, X^{(n)}_m$ be independent random matrices of order $n$ where the entries of 
$X^{(n)}_j$ are i.i.d. 
copies of $\sigma_j \frac{\xi}{\sqrt{n}}$ for some collection of positive constants $\sigma_1,\ldots,\sigma_m$.  Then the ESD $\mu_X^{(n)}$ of 
$X^{(n)} = X^{(n)}_1 X^{(n)}_2 \cdots X^{(n)}_m$ converges, with probability $1$, as $n\rightarrow \infty$ to the distribution whose density is given by 
\begin{equation}
\label{plotnost}
	\rho(z,\bar{z}) = \left\{
     			\begin{array}{lr}
       				\frac{1}{m \pi} \sigma^{-\frac{2}{m}} |z|^{\frac{2}{m}-2} & \text{ for } |z|\leq \sigma, \\
       				0 & \text{ for } |z| > \sigma,
     			\end{array} \right.
\end{equation}
where $\sigma = \sigma_1 \cdots \sigma_m$.
\end{theorem}

\begin{remark}
The almost sure convergence of $\mu_X^{(n)}$ implies the convergence of $\E \mu_X^{(n)}$ as well.
\end{remark}

\begin{remark}
We refer the reader to \cite{bai10} for bounds on powers of a square random matrix with i.i.d. entries.  See also \cite{alexeev1},
\cite{alexeev2}, \cite{burda1}, \cite{baimiao}, \cite{capitaine}, and \cite{speicher} for some other results on the spectral properties of 
products of random matrices.
\end{remark}

\section{Notation and Setup}

The proof of Theorem \ref{thm:main} is divided into two parts and presented in Sections \ref{section:proof} and \ref{section:circular}.  

We note that without loss of generality, we may assume $\sigma_1 = \sigma_2 = \cdots = \sigma_m = 1$.  Indeed, the spectrum for arbitrary $\sigma_1,\ldots,\sigma_m$ can be obtained by a trivial rescaling.  Following Burda, Janik, and Waclaw in \cite{burda}, we 
let $Y^{(n)}$ be a $(mn) \times (mn)$ matrix defined as
\begin{equation} \label{def-Y}
	Y^{(n)} = \left( \begin{array}{lllllc}
                         0 &  X^{(n)}_1 &     &            & 0       \\
                         0 & 0    & X^{(n)}_2 &            & 0        \\      
                           &      & \ddots & \ddots     &         \\
                         0 &      &     &          0 & X^{(n)}_{m-1} \\
                       X^{(n)}_m &      &     &            &  0      
	\end{array}\right).
\end{equation}

Section \ref{section:circular} will be devoted to proving that the ESD of $Y^{(n)}$ obeys the circular law as $n$ tends to infinity.  This statement is presented in the following Lemma.  

\begin{lemma} [$Y^{(n)}$ obeys the circular law] \label{lemma:Y-circular}
The ESD $\mu_{Y^{(n)}}$ of $Y^{(n)}$ converges, with probability $1$, to the uniform distribution over the unit disk as $n \rightarrow \infty$.
\end{lemma}

\section{Proof of Theorem \ref{thm:main}} \label{section:proof}

With Lemma \ref{lemma:Y-circular} above, we are ready to prove Theorem \ref{thm:main}.  

\begin{proof}[Proof of Theorems \ref{thm:main}]
Using the definition of $Y^{(n)}$ in \eqref{def-Y}, we can compute
\begin{equation*}
	\left(Y^{(n)}\right)^m = \left( \begin{array}{cccc}
                          Y_1 &      &  &         0 \\
                              &  Y_2 &  &          \\      
                              &      &  \ddots  &         \\
                           0  &      &          & Y_m 
                  
	\end{array}\right),
\end{equation*}
where $Y_k = X^{(n)}_k X^{(n)}_{k+1} \cdots X^{(n)}_m X^{(n)}_1 \cdots X^{(n)}_{k-1}$ for $1 \leq k \leq m$.  Notice that each $Y_k$ has the same eigenvalues as $X^{(n)}$.  Let $\lambda_1,\ldots,\lambda_n$ denote the eigenvalues of $X^{(n)}$ and let $\eta_1,\ldots,\eta_{mn}$ denote the eigenvalues of $Y^{(n)}$.  Then it follows that each $\lambda_k$ is an eigenvalue of $\left(Y^{(n)}\right)^m$ with multiplicity $m$.  

Let $f:\C \rightarrow \C$ be a continuous, bounded function.  Then we have
\begin{align*}
	\int_{\C} f(z) \ud \mu_{X^{(n)}}(z,\bar{z}) =& \frac{1}{n} \sum_{k=1}^n f(\lambda_k) = \frac{1}{mn} \sum_{k=1}^{mn} f(\eta^m_k) = \int_{\C}f(z^m) \ud \mu_{Y^{(n)}}(z, \bar{z}).
\end{align*}
By Lemma \ref{lemma:Y-circular},
\begin{equation*}
	\int_{\C}f(z^m) \ud \mu_{Y^{(n)}}(z, \bar{z}) \longrightarrow \frac{1}{\pi}\int_{\mathbb{D}}f(z^m) \ud z \ud \bar{z} \qquad \text{a.s.}
\end{equation*}
as $n \rightarrow \infty$ where $\mathbb{D}$ denotes the unit disk in the complex plane.  Thus, by the change of variables $z \mapsto z^{m}$ and $\bar{z} \mapsto \bar{z}^{m}$ we can write
\begin{align*}
		\frac{1}{\pi}\int_{\mathbb{D}}f(z^m) \ud z \ud \bar{z} = \frac{m}{\pi}\int_{\mathbb{D}}f(z) \frac{1}{m^2} \left|z\right|^{\frac{2}{m}-2} \ud z \ud \bar{z} = \frac{1}{m \pi} \int_{\mathbb{D}} f(z) \left|z\right|^{\frac{2}{m}-2} \ud z \ud \bar{z}.
\end{align*}
where the factor of $m$ out front of the integral corresponds to the fact that the transformation maps the complex plane $m$ times onto itself.  

Therefore, we have shown that for all continuous, bounded functions $f$, 
\begin{equation*}
	\int_{\C} f(z) \ud \mu_{X^{(n)}}(z,\bar{z}) \longrightarrow \frac{1}{m \pi} \int_{\mathbb{D}} f(z) \left|z\right|^{\frac{2}{m}-2} \ud z \ud \bar{z} \qquad \text{a.s.}
\end{equation*}
as $n \rightarrow \infty$ and the proof is complete. 
\end{proof}

\section{Proof of Lemma \ref{lemma:Y-circular}} \label{section:circular}

In order to prove that the ESD of $Y^{(n)}$ obeys the circular law, we follow the work of Bai in \cite{bai}, Bai and Silverstein in \cite{bai-book}, and 
use the results developed by Tao and Vu in \cite{tao}.  To do so, we introduce the following notation.  Let $\mu_n$ denoted the ESD of $Y^{(n)}$.  That is,
\begin{equation*}
	\mu_n(x,y) = \frac{1}{mn} \# \left\{k \leq mn : \Re(\lambda_k) \leq x; \Im(\lambda_k) \leq y\right\}
\end{equation*}
where $\lambda_1, \ldots, \lambda_{mn}$ are the eigenvalues of $Y^{(n)}$.  

An important idea in the proof is to analyze the Stieltjes transformation $s_n: \C \rightarrow \C$ of $\mu_n$ defined by 
\begin{equation*}
	s_n(z) = \frac{1}{mn} \sum_{k=1}^{mn} \frac{1}{\lambda_k -z} = \int_{\C} \frac{1}{x+iy-z} \ud \mu_n(x,y).
\end{equation*}
Since $s_n(z)$ is analytic everywhere except the poles, the real 
part determines the eigenvalues.  Let $z=s+it$.  Then we can write
\begin{align*}
	\Re(s_n(z)) =& \frac{1}{mn}\sum_{k=1}^{mn} \frac{\Re(\lambda_k)-s}{|\lambda_k - z|^2} \\
		=& -\frac{1}{2mn} \sum_{k=1}^{mn} \frac{\partial}{\partial s} \ln |\lambda_k - z|^2 \\
		=& -\frac{1}{2} \frac{\partial}{\partial s} \int_0^\infty \ln x \nu_n(\ud x, z)
\end{align*}
where $\nu_n(\cdot, z)$ is the ESD of the Hermitian matrix $H_n = (Y^{(n)}-zI)^\ast(Y^{(n)}-zI)$.  This reduces the task to controlling the 
distributions $\nu_n$.  

The main difficulties arise from the two poles of the log function, at $\infty$ and $0$.  We will need to use the bounds developed in \cite{bai} 
and \cite{tao} to control the largest singular value and the least singular value of $Y^{(n)}-zI$.  

A version of the following lemma was first presented by Girko, \cite{girko}.  We present a slightly refined version by Bai and Silverstein, 
\cite{bai-book}.  

\begin{lemma} \label{lemma:girko}
For any $uv \neq 0$, we have
\begin{align} 
	c_n(u,v) =& \int \int e^{iux + ivy} \mu_n(\ud x, \ud y) \nonumber \\
		\label{girko-identity}
		=& \frac{u^2 + v^2}{4iu\pi} \int \int \frac{\partial}{\partial s} \left[ \int_0^\infty \ln x \nu_n(\ud x, z) \right] e^{ius+ivt} \ud t \ud s,
\end{align}
where z=s+it.
\end{lemma}

We note that the singular values of $Y^{(n)}$ are the union of the singular values of $X^{(n)}_k$ for $1 \leq k\leq n$.  Thus, under the assumptions of Theorem \ref{thm:main}, the ESD of $Y^{(n)^\ast} Y^{(n)}$ converges to the Marchenko-Pastur Law (see \cite{mp} and \cite[Theorem 3.7]{bai-book}).  Thus by Lemma \ref{lemma:eigenvalue-singular-bound} it follows that, with probability $1$, the family of distributions $\mu_n$ is tight.  To prove the circular law we will show that the right-hand side of \eqref{girko-identity} converges to $c(u,v)$, its counterpart generated by the circular law, for all $uv \neq 0$.  Several steps of the proof will follow closely the work of Bai in \cite{bai} and Bai and Silverstein in \cite{bai-book}.  We present an outline of the proof as follows.  
\begin{enumerate}
\item We reduce the range of integration to a finite rectangle in Section \ref{section:integral}.  We will show that the proof reduces to showing that, for every large $A>0$ and small $\epsilon>0$,
\begin{align*}
	&\int \int_T \left[ \frac{\partial}{\partial s} \int_0^\infty \ln x \nu_n(\ud x, z) \right] e^{ius+ivt} \ud s \ud t \\
	\rightarrow & \int \int_T \left[ \frac{\partial}{\partial s} \int_0^\infty \ln x \nu(\ud x, z) \right] e^{ius+ivt} \ud s \ud t \\
\end{align*}
where $T=\{(s,t): |s| \leq A, |t| \leq A^3, |\sqrt{s^2+t^2}-1| \geq \epsilon\}$ and $\nu(x,z)$ is the limiting spectral distribution of the sequence of matrices $H_n = (Y^{(n)}-zI)^\ast(Y^{(n)}-zI)$.  
\item We characterize the limiting spectrum $\nu(\cdot, z)$ of $\nu_n(\cdot,z)$.
\item We establish a convergence rate of $\nu_n(\cdot,z)$ to $\nu(\cdot, z)$ uniformly in every bounded region of $z$.
\item Finally, we show that for a suitably defined sequence $\epsilon_n$, with probability $1$, 
\begin{align*}
	\limsup_{n \rightarrow \infty} \int \int_T \left| \int_{\epsilon_n}^\infty \ln x(\nu_n(\ud x, z) - \nu(\ud x, z)) \right| &= 0
\end{align*}
and
\begin{align*}
	\lim_{n \rightarrow \infty} \int_0^{\epsilon_n} \ln x \nu_n(\ud x, z) &= 0.
\end{align*}
\end{enumerate}

\subsection{Notation}

In this section, we introduce some notation that we will use throughout the paper.  

First, we will drop the superscript $(n)$ from the matrices $Y^{(n)}$, $X^{(n)}$, $X_1^{(n)}, \ldots,  X_m^{(n)}$ and simply write $Y$, $X$, $X_1,\ldots,X_m$.  

We write $R=Y-zI$ where $I$ is the identity matrix and $z=s+it \in \C$.  We will continue to let $H_n = (Y-zI)^\ast(Y-zI) = R^\ast R$ and have $\nu_n(x,z)$ denote the empirical spectral distribution of $H_n$ for each fixed $z$.  

For a $(mn) \times (mn)$ matrix $A$, there are $m^2$ blocks each consisting of a $n \times n$ matrix.  We let $A_{ab}$ denote the $n \times n$ matrix in position $a,b$ where $1 \leq a,b \leq m$.  \entry[A]{a}{b}{i}{j} then refers to the element $(A_{ab})_{ij}$ where $1 \leq i,j \leq n$.

Finally, $C$ will be used as some positive constant that may change from line to line.  

\subsection{Integral Range Reduction} \label{section:integral}

To establish Lemma \ref{lemma:Y-circular}, we need to find the limiting counterpart to
\begin{equation*}
	g_n(s,t) = \frac{\partial}{\partial s} \int_0^\infty \ln x \nu_n(\ud x, z).
\end{equation*}

We begin by presenting the following lemmas.

\begin{lemma}[Bai-Silverstein \cite{bai-book}]
For all $uv \neq 0$, we have
\begin{equation*}
	c(u,v) = \frac{1}{\pi} \int \int_{x^2 + y^2 \leq 1} e^{iux + ivy} \ud x \ud y = \frac{u^2 + v^2}{4iu\pi} \int \left[ \int g(s,t) e^{ius+ivt} \ud t \right] \ud s,
\end{equation*}
where 
\begin{equation*}
	g(s,t) =  \left\{
		\begin{array}{rl}
 		\frac{2s}{s^2+t^2}, & \text{if } s^2 + t^2 > 1\\
  		2s, & \text{otherwise} \\
		\end{array} \right.
\end{equation*}
\end{lemma}

\begin{lemma}[Horn-Johnson \cite{horn}] \label{lemma:eigenvalue-singular-bound}
Let $\lambda_j$ and $\eta_j$ denote the eigenvalues and singular values of an $n \times n$ matrix $A$, respectively.  Then for any $k \leq n$,
\begin{equation*}
	\sum_{j=1}^k |\lambda_j|^2 \leq \sum_{j=1}^k \eta^2
\end{equation*}
if $\eta_j$ is arranged in descending order.
\end{lemma}

\begin{lemma}[Bai-Silverstein \cite{bai-book}] \label{lemma:integral-bounds}
For any $uv \neq 0$ and $A > 2$, we have
\begin{equation*}
	\left| \int_{|s| \geq A} \int_{-\infty}^{\infty} g_n(s,t) e^{ius+ivt} \ud t \ud s \right| \leq \frac{4\pi}{|v|} e^{-\frac{1}{2}|v|A} + \frac{2\pi}{n |v|} \sum_{k=1}^{mn} I\left( |\lambda_k | \geq \frac{A}{2} \right)
\end{equation*}
and
\begin{equation*}
	\left| \int_{|s| \leq A} \int_{t \geq A^3} g_n(s,t) e^{ius+ivt} \ud t \ud s \right| \leq \frac{8A}{A^2-1} + \frac{4 \pi A}{n} \sum_{k=1}^{mn} I(|\lambda_k| > A)
\end{equation*}
where $\lambda_1, \ldots, \lambda_{mn}$ are the eigenvalues of $Y$.  Furthermore, if the function $g_n(s,t)$ is replaced by $g(s,t)$, the two inequalities above hold without the second terms.
\end{lemma}

Now we note that under the assumptions of Theorem \ref{thm:main} and by Lemma \ref{lemma:eigenvalue-singular-bound} and the law of large numbers we have 
\begin{equation*}
	\frac{1}{n} \sum_{k=1}^{mn} I(|\lambda_k| > A) \leq \frac{1}{nA^2} \tr(Y^\ast Y) \longrightarrow \frac{m}{A^2} \qquad \text{ a.s.}
\end{equation*}
Therefore, the right-hand sides of the inequalities in Lemma \ref{lemma:integral-bounds} can be made arbitrarily small by making $A$ large enough.  The same is true when $g_n(s,t)$ is replaced by $g(s,t)$.  Our task is then reduced to showing
\begin{equation*}
	\int_{|s| \leq A} \int_{|t| \leq A^3} [g_n(s,t) - g(s,t)] e^{ius+ivt} \ud s \ud t \longrightarrow 0.
\end{equation*}

We define the sets
\begin{equation*}
	T = \left\{ (s,t) : |s| \leq A, |t| \leq A^3 \text{ and } ||z|-1| \geq \epsilon \right\}
\end{equation*}
and
\begin{equation*}
	T_1 = \left\{(s,t) : ||z-1| < \epsilon \right\},
\end{equation*}
where $z=s+it$.  

\begin{lemma}[Bai-Silverstein \cite{bai-book}] \label{lemma:integral-esp-bound}
For all fixed $A$ and $0 < \epsilon < 1$,
\begin{equation} \label{integral-eps-bound}
	\int \int_{T_1} |g_n(s,t)| \ud s \ud t \leq 32 \sqrt{\epsilon}.
\end{equation}
Furthermore, if the function $g_n(s,t)$ is replaced by $g(s,t)$, the inequality above holds.  
\end{lemma}

Since the right-hand side of \eqref{integral-eps-bound} can be made arbitrarily small by choosing $\epsilon$ small, our task is reduced to showing 
\begin{equation} \label{task-reduction}
	\int \int_{T} [g_n(s,t) - g(s,t)] e^{ius+ivt} \ud s \ud t \longrightarrow 0 \qquad \text{a.s.}
\end{equation}

\subsection{Characterization of the Circular Law}

In this section, we study the convergence of the distributions $\nu_n(x,z)$ to a limiting distribution $\nu(x,z)$ as well as discuss properties of the limiting distribution $\nu(x,z)$.  We begin with a standard truncation argument which can be found, for example, in \cite{bai-book}.

\subsubsection{Truncation}

Let $\widehat{Y}$ and $\widetilde{Y}$ be the $(mn) \times (mn)$ matrices with entries
\begin{equation*}
	\entry[\widehat{Y}]{a}{b}{i}{j} = \entry[Y]{a}{b}{i}{j}I(\sqrt{n}|\entry[Y]{a}{b}{i}{j}| \leq n^\delta) - \E\entry[Y]{a}{b}{i}{j}I(\sqrt{n}|\entry[Y]{a}{b}{i}{j}| \leq n^\delta)
\end{equation*} 
and 
\begin{equation*}
	\entry[\widetilde{Y}]{a}{b}{i}{j} = \frac{\entry[\widehat{Y}]{a}{b}{i}{j}}{ \sqrt{ n \E \left|\entry[\widehat{Y}]{a}{b}{i}{j}\right|^2 }}
\end{equation*}
where $\delta>0$.  We denote the ESD of $\widehat{H}_n = (\widehat{Y}-zI)^\ast (\widehat{Y}-zI)$ by $\widehat{\nu}_n(\cdot, z)$ and the ESD of $\widetilde{H}_n = (\widetilde{Y}-zI)^\ast (\widetilde{Y}-zI)$ by $\widetilde{\nu}_n(\cdot, z)$.  

We will let $L(F_1,F_2)$ be the Levy distance between two distribution functions $F_1$ and $F_2$ defined by 
\begin{equation*}
	L(F_1,F_2) = \inf \{ \epsilon : F_1(x-\epsilon)-\epsilon \leq F_2(x) \leq F_1(x+\epsilon)+\epsilon \text{ for all } x \in \R \}.
\end{equation*}
We then have the following Lemma.

\begin{lemma}	\label{lemma:truncation}
We have that
\begin{equation*}
	L(\nu_n(\cdot, z), \widetilde{\nu}_n(\cdot, z)) = o(n^{-\eta \delta / 4}) \text{ a.s.}
\end{equation*}
where the bound is uniform for $|z| \leq M$.  
\end{lemma}

\begin{proof}
By \cite[Corollary A.42]{bai-book} we have that
\begin{equation*}
	L^4(\nu(\cdot, z), \widehat{\nu}_n(\cdot, z)) \leq \frac{2}{n^2}\tr(H_n - \widehat{H}_n) \tr[ (Y-\widehat{Y})^\ast (Y- \widehat{Y})].
\end{equation*}
By the law of large numbers it follows that, with probability $1$,
\begin{equation*}
	\frac{1}{n} \tr{H_n} = \frac{1}{n} \sum_{a=1}^m \sum_{1 \leq i, j \leq n} |\entry[Y]{a}{a+1}{i}{j}|^2 + m|z|^2 \longrightarrow m(1 + |z|^2).
\end{equation*}
Similarly, $\frac{1}{n} \tr(\widehat{H}_n) \rightarrow m(1+|z|^2)$ a.s.

For any $L>0$, we have  
\begin{align*}
	\frac{n^{\delta \eta}}{n}& \tr[(Y-\widehat{Y})^\ast (Y- \widehat{Y})] = \frac{n^{\delta \eta}}{n} \sum_{a=1}^m \sum_{1 \leq i,j \leq n} |\entry[(Y-\widehat{Y})]{a}{a+1}{i}{j}|^2 \\
		& \leq \frac{n^{\delta \eta}}{n^2} \sum_{a=1}^m \sum_{1 \leq i,j \leq n} \left| \sqrt{n}\entry[Y]{a}{a+1}{i}{j}I(\sqrt{n}|\entry[Y]{a}{a+1}{i}{j}| > n^\delta) - \E \sqrt{n}\entry[Y]{a}{a+1}{i}{j}I(\sqrt{n}|\entry[Y]{a}{a+1}{i}{j}|> n^\delta) \right|^2 \\
		& \leq 2n^{\delta \eta} \left( \frac{1}{n^2} \sum_{a=1}^m \sum_{1 \leq i,j \leq n} | \sqrt{n}\entry[Y]{a}{a+1}{i}{j}|^2 I(\sqrt{n}|\entry[Y]{a}{a+1}{i}{j}| > n^\delta) + \E |\xi|^2 I(|\xi|> n^\delta) \right) \\
		&\leq \frac{2}{n^2} \sum_{a=1}^m \sum_{1 \leq i,j \leq n} | \sqrt{n}\entry[Y]{a}{a+1}{i}{j}|^{2+\eta} I(\sqrt{n}|\entry[Y]{a}{a+1}{i}{j}| > L) + \E |\xi|^{2+\eta} I(|\xi|> L)
\end{align*}
and hence
\begin{equation*}
	\limsup_{n \rightarrow \infty} \frac{n^{\delta \eta}}{n} \tr[(Y-\widehat{Y})^\ast (Y- \widehat{Y})] \leq 4 m\E |\xi|^{2+\eta} I(|\xi|> L) \text{ a.s.} 
\end{equation*}
which can be made arbitrarily small by making $L$ large.  Thus we have that
\begin{equation*}
	L(\nu(\cdot, z), \widehat{\nu}_n(\cdot, z)) = o(n^{-\eta \delta / 4}) \text{ a.s.}
\end{equation*}
where the bound is uniform for $|z| \leq M$.  

By \cite[Corollary A.42]{bai-book} we also have that
\begin{equation*}
	L^4(\widehat{\nu}(\cdot, z), \widetilde{\nu}_n(\cdot, z)) \leq \frac{2}{n^2} \tr (\widehat{H}_n + \widetilde{H}_n) \tr(\widehat{Y}^\ast \widehat{Y}) \left( 1 - \frac{1}{\sqrt{ \E|\sqrt{n}\entry[\widehat{Y}]{1}{2}{1}{1}|^2 }} \right).
\end{equation*}
A similar argument shows that $1 - \sqrt{\E|\sqrt{n}\entry[\widehat{Y}]{1}{2}{1}{1}|^2} = o(n^{-\eta \delta})$ and the proof is complete.

\end{proof}

\begin{remark}
For the remainder of the subsection, we will assume the conditions of Theorem \ref{thm:main} hold.  Also, by Lemma \ref{lemma:truncation} we additionally assume that $| \entry[Y]{a}{a+1}{i}{j}| \leq n^\delta$.  
\end{remark}

\subsubsection{Useful tools and lemmas}

We begin by denoting the Stieltjes transform of ${\nu}_n(\cdot,z)$ by
\begin{equation*}
	\Delta_n(\alpha, z) = \int \frac{{\nu}_n(\ud x, z)}{x - \alpha},
\end{equation*}
where $\alpha= x + iy$ with $y>0$.  We also note that $\Delta_n(\alpha,z) = \frac{1}{mn} \tr({G})$ where ${G} = ({H}_n - \alpha I)^{-1}$ is the resolvent matrix.  For brevity, the variable $z$ will be suppressed when there is no confusion and we will simply write $\Delta_n(\alpha)$. 

We first present a number of lemmas that we will need to study $\Delta_n(\alpha)$.  We remind the reader that ${R} ={Y}-zI$ and $\alpha=x+iy$.  

\begin{lemma} \label{lemma:norm-bounds}
If $y>0$ and $x \in K$ for some compact set $K$, then we have the following bounds,
\begin{align}
	\|Y \|^2 &\leq \max_{1 \leq k \leq m} \|X_k \|^2 \leq \sum_{k=1}^m \|X_k \|^2, \label{Y-norm} \\ 
	\|G \|  &\leq \frac{1}{y}, \label{G-norm}\\
	\| RG \| & \leq C\sqrt{ \frac{1}{y^2} + \frac{1}{y}}, \label{RG-norm} \\
	\|G R^\ast \| &\leq C\sqrt{ \frac{1}{y^2} + \frac{1}{y}}, \label{GR-norm} 
\end{align}
for some constant $C>0$ which depends on $K$.  Moreover, there exists a constant $C$ which depends only on $K$ such that
\begin{align}
	\sup \left \{ \| RG \| : x \in K, y \geq y_n, z \in \C \right \} &\leq C\sqrt{ \frac{1}{y_n^2} + \frac{1}{y_n}}, \label{RG-unif-bound} \\
	\sup \left \{ \| G R^\ast \| : x \in K, y \geq y_n, z \in \C \right \} &\leq C\sqrt{ \frac{1}{y_n^2} + \frac{1}{y_n}}, \label{GR-unif-bound}
\end{align}
for any sequence $y_n > 0$.  
\end{lemma}

\begin{proof}
The first inequality in \eqref{Y-norm} follows from the definition of the norm and the second inequality is trivial.  The resolvent bound in \eqref{G-norm} follows immediately because $H_n$ is a Hermitian matrix.

To prove \eqref{RG-norm}, we use polar decomposition to write $R=U|R|$ where $U$ is a partial isometry and $|R| = \sqrt{R^\ast R}$.  Then
\begin{align*}
	\|RG \| &= \| U|R| (R^\ast R-\alpha)^{-1} \| \leq \| |R| (R^\ast R - \alpha)^{-1} \| \\
	&\leq \sup_{t \in \text{Sp}(R^\ast R)} | \sqrt{t}(t-\alpha)^{-1} | \leq \sup_{t \geq 0} | \sqrt{t}(t-\alpha)^{-1} | \leq C\sqrt{ \frac{1}{y^2} + \frac{1}{y}}.
\end{align*}
A similar argument verifies \eqref{GR-norm}.  \eqref{RG-unif-bound} and \eqref{GR-unif-bound} follow from \eqref{RG-norm} and \eqref{GR-norm} by using that $y \geq y_n$.  
\end{proof}

\begin{lemma} \label{lemma:G-distribution}
We have that
\begin{equation*}
	\E \left[ \frac{1}{n} \tr G_{a,a} \right] = \E \left[ \frac{1}{mn} \tr G \right]
\end{equation*}
for any $1 \leq a \leq m$.  
\end{lemma}
\begin{proof}
Fix $1 \leq a \leq m$ and $1 \leq i \leq n$.  We will show that 
\begin{equation*}
	\E\entry{a}{a}{i}{i} = \E\entry{a+1}{a+1}{i}{i}.
\end{equation*}
Using the adjoint formula for the inverse of a matrix, we can write that for any $1 \leq b \leq m$
\begin{equation*}
	\entry{b}{b}{i}{i} = \frac{ \det\left(R^\ast R - \alpha I\right)^{(b,i)}}{\det\left(R^\ast R - \alpha I \right)}
\end{equation*}
where $\left(R^\ast R - \alpha I\right)^{(b,i)}$ is the matrix $R^\ast R - \alpha I$ with the entries in the row and column that contain the element \entry[(R^\ast R - \alpha I)]{b}{b}{i}{i} replaced by zeroes except for the diagonal element which is replaced by a $1$.  

We will write $Q_b = X_b^\ast X_b + |z|^2 I - \alpha I$ and then note that $R^\ast R - \alpha I$ has the form
\begin{equation} \label{block-form}
	 \left( \begin{array}{cccccc}
				Q_m & - \bar{z}X_1 & 0 & \cdots & 0 & -z X_m^\ast \\
				-z X_1^\ast  & Q_1 & -\bar{z}X_2 & 0 & \cdots & 0 \\
				0 & -zX_2^\ast & Q_2 & \ddots & 0 & \vdots \\
				\vdots & 0 & \ddots & \ddots & -\bar{z}X_{m-2} & 0 \\
				0 & \cdots & 0 & -zX_{m-2}^\ast & Q_{m-2} & -\bar{z}X_{m-1} \\
				-\bar{z}X_m & 0 & \cdots & 0 & -z X_{m-1}^\ast & Q_{m-1}
	\end{array}\right),
\end{equation}
where $Q_m, Q_1, \ldots, Q_{m-1}$ appear along the diagonal.  

Let $\sigma = (1 \  2 \  3 \ldots m) \in S_m$.  We now construct two bijective maps.  Let $T_\sigma$ be the map that takes matrices of the form \eqref{block-form} into the matrix where each occurrence of $X_b$ is replaced by $X_{\sigma(b)}$ and each occurrence of $Q_b$ is replaced by $Q_{\sigma(b)}$.  Also, let 

\begin{equation*}
	\Omega = \underbrace{\C^{n^2} \times \C^{n^2} \times \cdots \times \C^{n^2}}_{m \text{ times}}
\end{equation*}
denote the probability space.  Then we write $\omega \in \Omega$ as $\omega = (X_1, X_2, \ldots, X_m)$.  We now define $T'_{\sigma}: \Omega \rightarrow \Omega$ by $T'_\sigma (X_1, \ldots, X_m) = (X_2, X_3, \ldots, X_m, X_1)$.  Since each $X_1, \ldots, X_m$ is an independent and identically distributed random matrix, $T'_\sigma$ is a measure preserving map.  

We claim that $\det(R^\ast R - \alpha I) = \det(T_\sigma(R^\ast R - \alpha I))$.  Indeed, if $\lambda$ is an eigenvalue of $(R^\ast R - \alpha I)$ with eigenvector $v = (v_m, v_1, \ldots, v_{m-1})\T$ where $v_b$ is an $n$-vector, then a simple computation reveals that $w=(v_{\sigma(m)}, v_{\sigma(1)}, \ldots, v_{\sigma(m-1)})\T = (v_1, \ldots, v_m)\T$ is an eigenvector of $T_\sigma(R^\ast R - \alpha I)$ with eigenvalue $\lambda$.  

Similarly, $\det\left(R^\ast R - \alpha I\right)^{(b,i)} = \det\left(T_\sigma \left(\left(R^\ast R - \alpha I\right)^{(b,i)}\right) \right)$.  Define $f_{a,i}(\omega)$ to be $\det\left(R^\ast R - \alpha I\right)^{(b,i)}(\omega)$ for each realization $\omega \in \Omega$.  Then we have that
\begin{align*}
	f_{a+1, i}(\omega) &= \det \left(R^\ast R- \alpha I \right)^{(a+1,i)}(\omega) \\
	&= \det \left(T_\sigma \left( \left(R^\ast R- \alpha I \right)^{(a+1,i)} \right) \right)(\omega) \\
	&= \det \left(T_\sigma \left(R^\ast R - \alpha I \right) \right)^{(a,i)} (\omega) \\
	&= \det \left(R^\ast R - \alpha I \right)^{(a,i)}(T'_\sigma (\omega)) = f_{a,i} (T'_\sigma(\omega))
\end{align*}
and
\begin{equation*}
	\det(R^\ast R- \alpha I)(\omega) = \det(R^\ast R- \alpha I)(T'_\sigma (\omega)).
\end{equation*}
Thus $\entry{a}{a}{i}{i}(T'_\sigma (\omega)) = \entry{a+1}{a+1}{i}{i}(\omega)$ for each $\omega \in \Omega$.  Since $T'_\sigma$ is measure preserving, the proof is complete.  
\end{proof}

Next, we present the decoupling formula, which can be found, for example, in \cite{kkp}.  If $\xi$ is a real-valued random variable such that $\E|\xi|^{p+2} < \infty$ and if $f(t)$ is a complex-valued function of a real variable such that its first $p+1$ derivatives are continuous and bounded, then
\begin{equation}
	\E[\xi f(\xi)] = \sum_{a=0}^p \frac{\kappa_{a+1}}{a!} \E[f^{(a)}(\xi)] + \epsilon, \label{decoupling}
\end{equation}
where $\kappa_a$ are the cumulants of $\xi$ and $|\epsilon| \leq C \sup_t |f^{(p+1)}(t)| \E|\xi|^{p+2}$ where $C$ depends only on $p$.  

If $\xi$ is a Gaussian random variable with mean zero, then all the cumulants vanish except for $\kappa_2$ and the decoupling formula reduces to the exact equation
\begin{equation*}
	\E[\xi f(\xi)] = \E[\xi^2] \E[f'(\xi)].
\end{equation*}

Finally, to use \eqref{decoupling}, we need to compute the derivatives of the resolvent matrix $G$ with respect to the various entries of $Y$.  This can be done by utilizing the resolvent identity and we find
\begin{align*}
	\frac{ \partial \entry{a}{b}{k}{l}}{\partial \Re( \entry[Y]{c}{c+1}{q}{p} )} &= - \entry[(GR^\ast)]{a}{c}{k}{q} \entry{c+1}{b}{p}{l} - \entry{a}{c+1}{k}{p} \entry[(RG)]{c}{b}{q}{l}, \\
	\frac{ \partial \entry{a}{b}{k}{l}}{\partial \Im( \entry[Y]{c}{c+1}{q}{p})} &= -i \entry[(GR^\ast)]{a}{c}{k}{q} \entry{c+1}{b}{p}{l} + i \entry{a}{c+1}{k}{p} \entry[(RG)]{c}{b}{q}{l}.
\end{align*}

\subsubsection{Main Theorem}

For the results below, we will consider $\alpha=x+iy$ where $y \geq y_n$ with $y_n = n^{-\eta \delta}$.  Our goal is to establish the following result.

\begin{theorem} \label{thm:cubic-eq}
Under the conditions of Theorem \ref{thm:main} and the additional assumption that $| \entry[Y]{a}{a+1}{i}{j}| \leq n^\delta$, we have
\begin{equation*}
	\Delta_n^3(\alpha,z) + 2 \Delta_n^2(\alpha,z) + \frac{ \alpha + 1 - |z|^2}{\alpha} \Delta_n(\alpha,z) + \frac{1}{\alpha} = r_n(\alpha, z),
\end{equation*}
where if $\delta$ is chosen such that $\delta \eta \leq 1/32$ and $\delta \leq 1/32$, then the remainder term $r_n$ satisfies
\begin{equation*}
	\sup \left \{ |r_n(\alpha, z)| : |z| \leq M, |x| \leq N, y \geq y_n \right\} = O\left( \delta_n \right) \qquad \text{a.s.}
\end{equation*}
with $\delta_n = n^{-1/4} y_n^{-5} n^\delta$.  
\end{theorem}

\begin{remark}
We note that the bounds presented here and in the rest of this section are not optimal and can be improved.  The bounds given, however, are sufficient for our purposes.  
\end{remark}

In order to prove Theorem \ref{thm:cubic-eq}, we will need the following lemmas.  The first lemma is McDiarmid's Concentration Inequality \cite{mcdiarmid}.

\begin{lemma}[McDiarmid's Concentration Inequality] \label{lemma:mcdiarmid}
Let $X = (X_1, X_2, \ldots, X_n)$ be a family of independent random variables with $X_k$ taking values in a set $A_k$ for each $k$.  Suppose that the real-valued $f$ defined on $\prod A_k$ satisfies
\begin{equation*}
	|f(x) - f(x')| \leq c_k
\end{equation*}
whenever the vectors $x$ and $x'$ differ only in the $k$th coordinate.  Let $\mu$ be the expected value of the random variable $f(X)$.  Then for any $t \geq 0$,
\begin{equation*}
	\Prob \left( |f(X) - \mu | \geq t \right) \leq 2 e^{-2t^2 / \sum c_k^2}.
\end{equation*}
\end{lemma}

\begin{remark}
McDiarmid's Concentration Inequality also applies to complex-valued functions by applying Lemma \ref{lemma:mcdiarmid} to the real part and imaginary part separately.  
\end{remark}

\begin{lemma}  \label{lemma:mcdiarmid-bound}
For $y \geq y_n$ and $|x| \leq N$ (where $\alpha=x+iy$), 
\begin{equation}
\label{lyon11}
	\Prob \left( | \Delta_n(\alpha,z) - \E \Delta_n(\alpha,z) | > t \right) \leq  4 e^{-ct^2 ny_n^4}
\end{equation}
for some absolute constant $c > 0$.   Moreover,  
\begin{equation*}
	\sup \left \{ |\Delta_n(\alpha,z) - \E\Delta_n(\alpha,z)| : |z| \leq M, |x| \leq N, y \geq y_n \right\} 
= O \left( n^{-1/4}y_n^{-2} \right) \qquad \text{a.s.}
\end{equation*}
\end{lemma}

\begin{proof}
Let $R_k$ denote the matrix $R$ with the $k$-th column replaced by zeroes.  Then $R^\ast R$ and $R^\ast_k R_k$ differ by a matrix with rank at most two.  
So by the resolvent identity
\begin{align} \label{resolv-bound}
	&\left| \frac{1}{mn} \tr \left( R^\ast R - \alpha \right)^{-1} - \frac{1}{mn} \tr \left( R^\ast_k R_k -\alpha \right)^{-1} \right| \nonumber \\
	& \qquad \leq \frac{2}{mn} \left\| \left( R^\ast R - \alpha \right)^{-1} \left( R^\ast_k R_k - R^\ast R \right) \left( R^\ast_k R_k -\alpha \right)^{-1} \right\| \\
	& \qquad \leq \frac{C}{n y_n} \sup_{t \geq 0} \left| (t-\alpha)^{-1}t \right| = C' \frac{1}{n y^2_n}  \nonumber
\end{align}
where the constant $C'$ depends only on $N$.  
The $mn$ columns of $Y$ form an independent family of random variables.  We now apply Lemma \ref{lemma:mcdiarmid} to the complex-valued function 
$\frac{1}{mn} \tr \left( R^\ast R - \alpha \right)^{-1}$ with the bound $c_k = O(n^{-1} y^{-2}_n)$ obtained in \eqref{resolv-bound}.  This proves
the bound (\ref{lyon11}). Thus, for any fixed point $(\alpha, z)$ in the region
\begin{equation}
\label{novreg}
\{ (\alpha=x+iy, z=s+it): |x| \leq N, y \geq y_n,  |z| \leq M \}
\end{equation}
one has
\begin{equation}
\label{lyon12}
	\Prob \left( | \Delta_n(\alpha,z) - \E \Delta_n(\alpha,z) | >  n^{-1/4} \* y_n^{-2} \right) \leq 4 e^{-c\*n^{1/2}}, 
\end{equation}
where we recall that $y_n =n^{-\eta\*\delta}$ and $\delta>0$ could be chosen to be arbitrary small.

If $y=\Im \alpha> n^{1/4}\*y_n^2,$ then 
\begin{equation}
\label{lyon13}
|\Delta_n(\alpha,z)|\leq \frac{1}{\Im \alpha}< n^{-1/4} \* y_n^{-2}, \ \ |\E\Delta_n(\alpha,z)|< n^{-1/4} \* y_n^{-2}.
\end{equation}
Therefore, it is enough to bound the supremum of $|\Delta_n(\alpha) - \E \Delta_n(\alpha) |$ over the region
\begin{equation}
\label{novnovreg}
\mathcal{D}=\{ (\alpha=x+iy, z=s+it): |x| \leq N, \ y_n \leq y \leq n^{1/4}\*y_n^2,  \ |z| \leq M \}.
\end{equation}
To this end, we consider a finite $n^{-C}$-net of $\mathcal{D}$ where $C$ is some sufficiently large positive constant to be chosen later.
Clearly, one can construct such a net that contains at most $[4 \*M \* n^{4\*C} \*n^{1/4}\*y_n^2] $ points if $n$ is sufficiently large, where $[k] $ 
denotes the integer part of $k.$ Let us denote these points by
$(\alpha_i, z_i), 1\leq i  \leq [4\*M  \* n^{4\*C} \*n^{1/4}\*y_n^2].$
It follows from (\ref{lyon12}) that 
one has
\begin{equation}
\label{lyon14}
\Prob \left( \sup \{i:  | \Delta_n(\alpha_i, z_i) - \E \Delta_n(\alpha_i, z_i) | >  n^{-1/4} \* y_n^{-2} \right) \leq 
16\*M \* y_n^2 \* n^{4C +1/4}\*e^{-c\*n^{1/2}},
\end{equation}
where the supremum is taken over the points of the net.
Appying the Borel-Cantelli lemma, we obtain that
\begin{equation}
\label{lyon15}
	\sup \left \{ i:  |\Delta_n(\alpha_i,z_i) - \E\Delta_n(\alpha_i,z_i)| \right\} 
= O \left( n^{-1/4}y_n^{-2} \right) \qquad \text{a.s.}
\end{equation}
where the supremum is again taken over the points of the $n^{-C}$-net of $\mathcal{D}.$
To extend the estimate (\ref{lyon15}) to the supremum over the whole region $\mathcal{D},$ we note that for $(\alpha, z) \in \mathcal{D},$
\begin{align}
\label{brest11}
& \left|\frac{ \partial \Delta_n(\alpha,z)}{\partial \Re \alpha}\right|\leq \frac{1}{y_n^2}, \\
\label{brest12}
& \left|\frac{ \partial \Delta_n(\alpha,z)}{\partial \Im \alpha}\right|\leq \frac{1}{y_n^2}, \\
\label{brest13}
& \left|\frac{ \partial \Delta_n(\alpha,z)}{\partial \Re z}\right|\leq const_m\frac{2\*(n^{1+\delta}+M)}{y_n^2}, \\
\label{brest14}
& \left|\frac{ \partial \Delta_n(\alpha,z)}{\partial \Im z}\right|\leq const_m\frac{2\*(n^{1+\delta}+M)}{y_n^2},
\end{align}
where $const_m$ is a constant that depends only on $m$.

The bounds (\ref{brest11}-\ref{brest12}) are simple properties of the Stieltjes transform.  Indeed, the l.h.s. of (\ref{brest11}) and 
(\ref{brest12}) are bounded from above by $\frac{1}{|\Im \alpha|^2}.$  The proof of (\ref{brest13}-\ref{brest14}) follows from the resolvent identitity
\begin{equation*}
(H_n(z_2) -\alpha \*I)^{-1}-(H_n(z_1) -\alpha \*I)^{-1}= (H_n(z_1) -\alpha \*I)^{-1} \*(H_n(z_2)-H_n(z_1)) \*
(H_n(z_2) -\alpha \*I)^{-1},
\end{equation*}
the formula $ H_n(z)=(Y^{(n)}-z\*I)^*(Y^{(n)}-z\*I),$
the bound $|z|\leq M, $
and the bound
\begin{equation}
\label{sena}
\|Y^{(n)}\| \leq n^{1+\delta}.
\end{equation}
We note that (\ref{sena}) follows from
the fact that the matrix entries of $Y^{(n)}$ are bounded by $n^{\delta}.$

Now, choosing $C$ in the construction of the net sufficiently large, one extends the bound (\ref{lyon15}) to the whole region $\mathcal{D}$ by
(\ref{brest11}-\ref{brest14}).
This finishes the proof of the lemma.  
\end{proof}

\begin{lemma} \label{lemma:variance-bound}
For any $1 \leq a \leq m$, 
\begin{equation*}
	\sup \left\{ \var \left( \frac{1}{n} \tr G_{aa} \right) : |x| \leq N, y\geq y_n, z \in \C \right\} = O( n^{-1}y_n^{-2} )
\end{equation*}
where $\alpha=x+iy$.
\end{lemma}

\begin{proof}
Let $R_k$ denote the matrix $R$ with the $k$-th column replaced by zeroes and let $P_a$ be the orthogonal projector such that 
$\tr G_{a,a} = \tr (P_a G P_a)$.  Following the same procedure as in the proof of Lemma \ref{lemma:mcdiarmid-bound}, we have that
\begin{align} 
	&\left| \tr (R^\ast R - \alpha)^{-1}_{a,a} - \tr (R^\ast_k R_k - \alpha)^{-1}_{a,a} \right| \nonumber \\
	& \qquad = \left| \tr \left[ P_a (R^\ast_k R_k - \alpha)^{-1} (R^\ast_k R_k - R^\ast R)(R^\ast R - \alpha)^{-1} P_a \right] \right| \label{resolvent-id-bound-var} \\
	& \qquad \leq \frac{C}{y^2_n}  \nonumber.
\end{align}
where the constant $C$ depends only on $N$.  

We can write
\begin{equation*}
	\frac{1}{n} \tr G_{a,a} - \E\left[ \frac{1}{n} \tr G_{a,a} \right] = \frac{1}{n} \sum_{k=1}^{mn} \gamma_k,
\end{equation*}
where $\gamma_k$ is the martingale difference sequence
\begin{equation*}
	\gamma_k = \E_k \left[\frac{1}{n} \tr G_{a,a} \right] - \E_{k-1} \left[ \frac{1}{n} \tr G_{a,a} \right]
\end{equation*}
and $\E_k$ denotes the conditional expectation with respect to the elements in the first $k$ columns of $Y$.  Then by the bound 
in \eqref{resolvent-id-bound-var} and \cite[Lemma 2.12]{bai-book}, we have
\begin{align*}
	\E \left| \frac{1}{n} \tr G_{a,a} - \E \frac{1}{n} \tr G_{a,a} \right|^2 &\leq \frac{C}{n^2} \sum_{k=1}^{mn} | \gamma_k |^2 \\
		& \leq \frac{C}{n y^2_n}
\end{align*}
where the constant $C$ depends only on $N$.  Since the bound holds for any $|x| \leq N, y \geq y_n$, and $z \in \C$ the proof is complete.    
\end{proof}

\begin{remark} \label{rem:variance-split}
By Lemmas \ref{lemma:G-distribution}, \ref{lemma:mcdiarmid-bound}, and \ref{lemma:variance-bound}, for every $1 \leq a,b,c \leq m$
\begin{align*}
	\E\left[ \frac{1}{n} \tr G_{a,a} \right] &= \E\left[\frac{1}{mn} \tr G\right] =\frac{1}{mn} \tr G + O(n^{-1/4}y_n^{-5}) \qquad \text{a.s.,} \\
	\E\left[ \frac{1}{n} \tr G_{a,a} \frac{1}{n} \tr G_{b,b} \right] &= \E\left[\left(\frac{1}{mn} \tr G \right)^2\right] + O(n^{-1/4}y_n^{-5}) \\ &=\left(\frac{1}{mn} \tr G \right)^2 + O(n^{-1/4}y_n^{-5}) \qquad \text{a.s.,} 
\end{align*}
and
\begin{align*}
	\E\left[ \frac{1}{n} \tr G_{a,a} \frac{1}{n} \tr G_{b,b} \frac{1}{n} \tr G_{c,c} \right] &= \E\left[ \left(\frac{1}{mn} \tr G \right)^3\right] + O(n^{-1/4}y_n^{-5})\\ &= \left(\frac{1}{mn} \tr G \right)^3 + O(n^{-1/4}y_n^{-5}) \qquad \text{a.s.,} \\
\end{align*}
where the bounds hold uniformly in the region $|x| \leq N, y \geq y_n$, and $|z| \leq M$.  
\end{remark}

We are now ready to prove Theorem \ref{thm:cubic-eq}. 

\begin{proof}[Proof of Theorem \ref{thm:cubic-eq}]
Fix $\alpha = x + iy$ with $|x| \leq N, y \geq y_n$ and $z \in \C$ with $|z| \leq M$.  We will show that the remainder term $r_n(\alpha, z) = O(\delta_n)$ a.s. where the constants in the term $O(\delta_n)$ depend only on $N$ and $M$.  In particular, the remainder term will be estimated using Lemmas \ref{lemma:norm-bounds} and \ref{lemma:mcdiarmid-bound} and Remark \ref{rem:variance-split} where the bounds all hold uniformly in the region.  In the proof presented below, will use the notation $O_{N,M}(\cdot)$ to represent a term which is bounded uniformly in the region $|x| \leq N, y \geq y_n$, and $|z| \leq M$.  

By applying the resolvent identity to $G$ and replacing $R$ and $R^\ast$ with $Y-zI$ and $Y^\ast - \bar{z}I$, respectively, we obtain
\begin{equation*}
	\frac{1}{n} \tr G_{a,a} = -\frac{1}{\alpha} + \frac{1}{\alpha n} \tr [GY^\ast Y]_{a,a} - \frac{z}{\alpha n} \tr[G Y^\ast]_{a,a} - \frac{\bar{z}}{\alpha n} \tr [GY]_{a,a} + \frac{|z|^2}{\alpha n} \tr G_{a,a}.
\end{equation*}

We will let $Y^{(r)}$ be the $(mn) \times (mn)$ matrix containing the real entries of $Y$ and $Y^{(i)}$ be the $(mn) \times (mn)$ matrix containing the imaginary entries of $Y$ such that $Y = Y^{(r)} + i Y^{(i)}$.  By assumption, $Y^{(r)}$ and $Y^{(i)}$ are independent random matrices.  Thus,
\begin{align} 
	\left(1 - \frac{|z|^2}{\alpha}\right) \E\frac{1}{n} \tr G_{a,a} + \frac{1}{\alpha} &= \frac{1}{\alpha n} \E \tr (GY^\ast Y^{(r)})_{a,a} + \frac{i}{\alpha n} \E \tr(G Y^\ast Y^{(i)})_{a,a} \nonumber \\
	& \qquad - \frac{z}{\alpha n} \E \tr(G {Y^{(r)}}^\ast)_{a,a} + \frac{iz}{\alpha n} \E \tr (G {Y^{(i)}}^\ast)_{a,a} \label{resolvent-expansion}\\
	& \qquad -\frac{\bar{z}}{\alpha n} \E \tr(GY^{(r)})_{a,a} - \frac{\bar{z}i}{\alpha n} \E \tr(G Y^{(i)})_{a,a} \nonumber 
\end{align}
Let $\delta=\var(\Re(\xi))$.  Then $\var(\Im(\xi)) = 1-\delta$.  To compute the expectation, we fix all matrix entries except one and integrate with respect to that entry.  Thus, by applying the decoupling formula \eqref{decoupling} with $p=1$ and using the fact that $\entry[Y]{a}{b}{i}{j} = 0$ whenever $b \neq a+1$, we obtain the following expansions for the terms on the right-hand side of \eqref{resolvent-expansion},
\begin{align*}
	\frac{1}{\alpha n} \E \tr (GY^\ast Y^{(r)})_{a,a} &= \frac{1}{\alpha n} \E \sum_{1 \leq j,k,l \leq n} \entry{a}{a}{j}{k} \entry[\bar{Y}]{a-1}{a}{l}{k} \Re \left( \entry[Y]{a-1}{a}{l}{j} \right) \\
	&= \frac{\delta}{\alpha n} \E \tr G_{a,a} - \frac{\delta}{\alpha n^2} \E\sum_{1 \leq j,k,l \leq n} \entry[\bar{Y}]{a-1}{a}{l}{k} \left( \entry[(GR^\ast)]{a}{a-1}{j}{l} \entry{a}{a}{j}{k}\right) \\
	& \qquad - \frac{\delta}{\alpha n^2} \E\sum_{1 \leq j,k,l \leq n} \entry[\bar{Y}]{a-1}{a}{l}{k} \left( \entry{a}{a}{j}{j} \entry[(RG)]{a-1}{a}{l}{k}\right) + O_{N,M} \left(\frac{n^\delta}{n^{1/2}y_n^4}\right) \\
	&= \frac{\delta}{\alpha n} \E \tr G_{a,a} - \frac{\delta}{\alpha n^2} \E \left[ \tr G_{a,a} \tr (RGY^\ast)_{a-1, a-1} \right] + O_{N,M} \left( \frac{n^\delta}{n^{1/2}y_n^4} \right).
\end{align*}
Here we use that the $\epsilon$ error term in \eqref{decoupling} contains the second derivative
\begin{equation*}
	\frac{ \partial^2 \entry[(GY^\ast)]{a}{a}{j}{l} }{\partial \Re \left( \entry[Y]{a-1}{a}{l}{j} \right)^2}  = O_{N,M} \left( \frac{1}{y_n^3} \right)
\end{equation*}
which consists of several terms each bounded by Lemma \ref{lemma:norm-bounds}.  After summing over $1 \leq j,l \leq n$ and utilizing the fact that the third moment of $\Re \left(\entry[Y]{a-1}{a}{l}{j} \right)$ is of order $n^{\delta-3/2}$, we obtain an error bound of $O_{N,M}(n^{\delta-1/2}y_n^{-4})$.  By following the same procedure for the other terms, we obtain
\begin{align*}
	\frac{i}{\alpha n} \E \tr(G Y^\ast Y^{(i)})_{a,a} &= \frac{i}{\alpha n} \E \sum_{1 \leq j,kl \leq n} \entry{a}{a}{j}{k} \entry[\bar{Y}]{a-1}{a}{l}{k} \Im \left( \entry[Y]{a-1}{a}{l}{j} \right) \\
	&= \frac{1-\delta}{\alpha n} \E \tr G_{a,a} - \frac{1-\delta}{\alpha n^2} \E \left[ \tr G_{a,a} \tr (RGY^\ast)_{a-1,a-1}\right] + O_{N,M} \left( \frac{n^\delta}{n^{1/2}y_n^4} \right) , \\
	\frac{z}{\alpha n} \E \tr(G {Y^{(r)}}^\ast)_{a,a} &= \frac{z}{\alpha n} \E \sum_{1 \leq j,k \leq n} \entry{a}{a}{j}{k} \Re\left(\entry[Y]{a}{a+1}{j}{k} \right) \\
	&= \frac{z \delta}{\alpha n^2} \E \left[ \tr G_{a+1}{a+1} \tr(GR^\ast)_{a,a} \right] + O_{N,M} \left( \frac{n^\delta}{n^{1/2}y_n^4} \right), \\
	\frac{iz}{\alpha n} \E \tr (G {Y^{(i)}}^\ast)_{a,a} &= \frac{iz}{\alpha n} \E \sum_{1 \leq j,k \leq n} \entry{a}{a+1}{j}{k} \Im\left( \entry[Y]{a}{a+1}{j}{k} \right) \\
	&= \frac{z(1-\delta)}{\alpha n^2} \E \left[ \tr G_{a+1,a+1} \tr (GR^\ast)_{a,a} \right] + O_{N,M} \left( \frac{n^\delta}{n^{1/2}y_n^4} \right), \\
	\frac{\bar{z}}{\alpha n} \E \tr(GY^{(r)})_{a,a} &= \frac{\bar{z}}{\alpha n} \E \sum_{1 \leq j,k \leq n} \entry{a}{a-1}{j}{k} \Re \left( \entry[Y]{a-1}{a}{k}{j} \right) \\
	&= \frac{\bar{z}\delta}{\alpha n^2} \E \left[ \tr G_{a,a} \tr (RG)_{a-1,a-1} \right] + O_{N,M} \left( \frac{n^\delta}{n^{1/2}y_n^4} \right), 
\end{align*}
and
\begin{align*}
	\frac{\bar{z}i}{\alpha n} \E \tr(G Y^{(i)})_{a,a} &= \frac{i \bar{z}}{\alpha n} \E \sum_{1 \leq j,k \leq n} \entry{a}{a-1}{j}{k} \Im\left( \entry[Y]{a-1}{a}{k}{j} \right) \\
	&= \frac{\bar{z}(1-\delta)}{\alpha n^2} \E \left[ \tr G_{a,a} \tr (RG)_{a-1,a-1} \right] + O_{N,M} \left( \frac{n^\delta}{n^{1/2}y_n^4} \right).
\end{align*}
Combining these terms yields,
\begin{align*}
	\left(1 - \frac{|z|^2}{\alpha}\right) \E\frac{1}{n} \tr G_{a,a} &= - \frac{1}{\alpha} + \frac{1}{\alpha n} \E \tr G_{a,a} - \frac{1}{\alpha n^2} \E \left[ \tr G_{a,a} \tr (RGR^\ast)_{a-1,a-1} \right] \\
	& \qquad + \frac{z}{\alpha n^2} \E \left[ \tr G_{a+1,a+1} \tr (GR^\ast)_{a,a} \right] + O_{N,M} \left( \frac{n^\delta}{n^{1/2}y_n^4} \right) \\
	&= - \frac{1}{\alpha} + \frac{1}{\alpha n} \E \tr G_{a,a} - \frac{1}{\alpha n^2} \E \left[ \tr G_{a,a} \tr (RGR^\ast)_{a-1,a-1} \right] \\
	& \qquad + \frac{z}{\alpha n^2} \E \left[ \tr G_{a+1,a+1} \tr (GR^\ast R)_{a,a} \right] \\
	& \qquad + \frac{1}{\alpha n^2} \E \left[ \tr G_{a+1,a+1} \tr(GR^\ast Y)_{a,a} \right] +O_{N,M} \left( \frac{n^\delta}{n^{1/2}y_n^4} \right).
\end{align*}
We note that by Remark \ref{rem:variance-split}, we have that
\begin{align*}
	\frac{-1}{\alpha n^2} \E\left[ \tr G_{a+1,a+1} \tr (GR^\ast R)_{a,a} \right] &= -\frac{1}{\alpha n} \E \tr G_{a+1,a+1} \\
	& \qquad- \frac{1}{mn^2} \E \tr G_{a,a} \E \tr G + O_{N,M}(\delta_n)
\end{align*}
and
\begin{align*}
	\frac{-1}{\alpha n^2} \E\left[ \tr G_{a,a} \tr (RGR^\ast)_{a-1,a-1} \right] &= -\frac{1}{\alpha n^2} \E \left[ \tr G_{a,a} \tr (R^\ast RG)_{a,a} \right] \\
	& \qquad - \frac{|z|^2}{\alpha n^2} \E \left[ \tr G_{a,a} ( \tr G_{a,a} - \tr G_{a-1,a-1}) \right] \\
	&= -\frac{1}{\alpha n} \E \tr G_{a,a} - \frac{1}{mn^2} \E \tr G_{a,a} \E \tr G + O_{N,M}(\delta_n).
\end{align*}

Finally, we expand $Y$ in terms of $Y^{(r)}$ and $Y^{(i)}$ and again apply the decoupling formula \eqref{decoupling} to obtain
\begin{align*}
	\frac{1}{\alpha n^2} \E \left[ \tr G_{a+1,a+1} \tr (GR^\ast Y)_{a,a} \right] &= -\frac{1}{n^3} \E \left[ \tr G_{a+1,a+1} \tr G_{a,a} \tr G_{a,a} \right] + O_{N,M} \left( \frac{n^\delta}{n^{1/2}y_n^4} \right) \\
	&= -\frac{1}{n^3 m^2} \left( \E \tr G \right)^2 \E \tr G_{a,a} + O_{N,M}(\delta_n),
\end{align*}
where the last equality comes from Remark \ref{rem:variance-split}.  Therefore, we have that
\begin{align*}
	\left(1 - \frac{|z|^2}{\alpha}\right) \E\frac{1}{n} \tr G_{a,a} &= - \frac{1}{\alpha} -\frac{1}{\alpha n} \E \tr G_{a+1,a+1} - \frac{2}{mn^2} \E \tr G_{a,a} \E \tr G \\
	& \qquad- \frac{1}{n^3 m^2} \left( \E \tr G \right)^2 \E \tr G_{a,a} + O_{N,M}(\delta_n).
\end{align*}
By summing over $a$ and dividing by $m$, we obtain
\begin{align*}
	\left( \E \Delta_n(\alpha) \right)^3 + 2 \left( \E \Delta_n(\alpha) \right)^2 + \frac{1 + \alpha - |z|^2}{\alpha} \E \Delta_n(\alpha) + \frac{1}{\alpha} = O_{N,M}(\delta_n).
\end{align*}
Thus, the proof is complete by Lemma \ref{lemma:mcdiarmid-bound}.
\end{proof}

Consider the cubic equation 
\begin{equation} \label{cubic-eq}
	\Delta^3 + 2 \Delta^2 + \frac{\alpha+1-|z|^2}{\alpha} \Delta + \frac{1}{\alpha} = 0
\end{equation}
where $\alpha = x + iy$.  The solution of the equation has three analytic branches when $\alpha \neq 0$ and when there is no multiple root.  Below we show 
that the Stieltjes transform of $\nu_n(\cdot, z)$ converges to a root of \eqref{cubic-eq}.  Following the argument of Bai and Silverstein in 
\cite{bai-book}, we have that there is only one of the three analytic branches, denoted by $\Delta(\alpha)$, to which the Stieltjes transforms are 
converging to.  We let $m_2(\alpha)$ and $m_3(\alpha)$ denote the other two branches and note that $\Delta$, $m_2$, and $m_3$ are also functions of $|z|$.  

By \cite[Theorem B.9]{bai-book}, there exists a distribution function $\nu(\cdot, z)$ such that 
\begin{equation*}
	\Delta(\alpha) = \int \frac{1}{u-\alpha} \nu(\ud u, z).
\end{equation*}
Then we use the following Lemmas due to Bai and Silverstein, \cite{bai-book}.

\begin{lemma}\label{lemma:nu-prop}
The limiting distribution function $\nu(x,z)$ satisfies
\begin{equation*}
	| \nu(w+u, z) - \nu(w,z)| \leq \frac{2}{\pi} \max \{ 2\sqrt{3|u|}, |u| \}
\end{equation*}
for all $z$.  Also, the limiting distribution function $\nu(u,z)$ has support in the interval $[x_1, x_2]$ when $|z| > 1$ and 
$[0, x_2]$ when $|z| \leq 1$, where
\begin{align*}
	x_1 &= \frac{1}{8 |z|^2} \left[ -1 + 20 |z|^2 + 8|z|^4 - (\sqrt{1+8|z|^2})^3 \right], \\
	x_2 &= \frac{1}{8 |z|^2} \left[ (\sqrt{1+8|z|^2})^3 - 1 + 20|z|^2 + 8 |z|^4 \right].
\end{align*}
\end{lemma}

\begin{lemma}\label{lemma:branch-bounds}
For any given constants $N>0$, $A>0$, and $\epsilon \in (0,1)$ (recall that $A$ and $\epsilon$ are used to define the region $T$), there exist positive 
constants $\epsilon_1$ and $\epsilon_0$ ($\epsilon_0$ may depend on $\epsilon_1$) such that for all large $n$, 
\begin{enumerate}[(i)]

\item for $|\alpha| \leq N$, $y \geq 0$, and $z \in T$,
\begin{equation*}
	\max_{j=2,3} | \Delta(\alpha) - m_j(\alpha)| \geq \epsilon_0, 
\end{equation*}

\item for $|\alpha| \leq N$, $y \geq 0$, $|\alpha - x_2| \geq \epsilon_1$ (and $|\alpha-x_1| \geq \epsilon_1$ if $|z| \geq 1 + \epsilon$), and $z \in T$,
\begin{equation*}
	\min_{j=2,3} | \Delta(\alpha) - m_j(\alpha)| \geq \epsilon_0, 
\end{equation*}

\item for $z \in T$ and $|\alpha-x_2| < \epsilon_1$,
\begin{equation*}
	\min_{j=2,3} | \Delta(\alpha) - m_j(\alpha)| \geq \epsilon_0 \sqrt{|\alpha - x_2|},
\end{equation*}

\item for $|z| > 1 + \epsilon$, $z \in T$, and $|\alpha-x_1| < \epsilon_1$,
\begin{equation*}
	\min_{j=2,3} | \Delta(\alpha) - m_j(\alpha)| \geq \epsilon_0 \sqrt{|\alpha - x_1|}.
\end{equation*}
\end{enumerate}
\end{lemma}

\begin{remark}
Lemma \ref{lemma:branch-bounds} shows that away from the real line, $\Delta$ is distinct from the branches $m_2$ and $m_3$.  
\end{remark}

\begin{lemma} \label{lemma:nu-circular}
We have 
\begin{equation*}
	\frac{\partial}{\partial s} \int_{0}^\infty \ln x \nu(\ud x, z) = g(s,t).
\end{equation*}
\end{lemma}

\begin{remark}
Lemma \ref{lemma:nu-circular} shows that $\nu(\cdot, z)$ is the distribution which corresponds to the circular law.  
\end{remark}

\subsection{Rate of Convergence of $\nu_n(x,z)$}
For this subsection, we return to the original assumptions on the entries of $Y$.  Before we prove Lemma \ref{lemma:Y-circular}, we need to establish a rate of convergence of $\nu_n(x,z)$ to $\nu(x,z)$.  We remind the reader that $\nu_n(\cdot, z)$ is the ESD of $H_n = (Y-zI)^\ast (Y-zI)$ and $\widetilde{\nu}_n(\cdot, z)$ is the ESD of $\widetilde{H}_n = (\widetilde{Y}-zI)^\ast (\widetilde{Y}-zI)$.  

\begin{lemma} \label{lemma:rate}
For any $M_2 > M_1 \geq 0$, 
\begin{equation*}
	\sup_{M_1 \leq |z| \leq M_2} \| \nu_n(\cdot, z) - \nu(\cdot, z) \| = \sup_{x, M_1 \leq |z| \leq M_2} | \nu_n(x, z) - \nu(x,z)| = O(n^{-\delta \eta / 8}).
\end{equation*}
\end{lemma}
\begin{proof}
We first note that it is enough to show 
\begin{equation} \label{truncate-esd-bound}
	\sup_{M_1 \leq |z| \leq M_2} \| \widetilde{\nu}_n(\cdot, z) - \nu(\cdot, z) \| = O(\sqrt{y_n}).  
\end{equation}
Indeed, by Lemma \ref{lemma:truncation}, we have that
\begin{align*}
	L(\nu_n(\cdot, z), \nu(\cdot, z)) &\leq L(\nu_n(\cdot, z), \widetilde{\nu}_n(\cdot, z)) + \| \widetilde{\nu}_n(\cdot, z) - \nu(\cdot, z) \| \\
	&\leq \| \widetilde{\nu}_n(\cdot, z) - \nu(\cdot, z) \| +  o(n^{-\eta \delta / 4}).
\end{align*}
and by Lemma \ref{lemma:nu-prop},
\begin{equation*}
	\|{\nu}_n(\cdot, z) - \nu(\cdot, z) \| \leq C\sqrt{ L(\nu_n(\cdot, z), \nu(\cdot, z)) }
\end{equation*}
uniformly for $|z| \leq M$.  

We now prove \eqref{truncate-esd-bound}.  Since $|\Delta_n(\alpha_0)| \leq (\Im \alpha_0)^{-1}$ for any fixed $\alpha_0$ with $\Im \alpha_0 > 0$, there exists a convergent subsequence of $\{\Delta_n(\alpha_0)\}_{n=1}^\infty$.  Since $\Delta$ is the only branch of \eqref{cubic-eq} that defines a Stieltjes transform, the subsequence must converge to $\Delta(\alpha_0)$.  Hence, $\Delta_n(\alpha_0) \rightarrow \Delta(\alpha_0)$ as $n\rightarrow \infty$ for any fixed $\alpha_0$ with $\Im \alpha_0 > 0$.  Let $m_1 = \Delta$ and $m_2$ and $m_3$ be the other two branches of the cubic equation \eqref{cubic-eq}.  

We remind the reader that $T$ is a bounded set and that the supports of $\nu(\cdot,z)$ are bounded for all $z \in T$.  So by \cite[Corollary B.15]{bai-book} there exists $N$ and some absolute constant $C$ such that
\begin{align*}
	& \| \widetilde{\nu}_n(\cdot,z) - \nu(\cdot,z) \| \\
	& \leq C \left( \int_{|x| \leq N} | \Delta_n(\alpha) - \Delta(\alpha) | \ud x + \frac{1}{y_n} \sup_{x} \int_{|y| \leq 2y_n} |\nu(x+y,z) - \nu(x,z)| \ud y \right) \\
	& \leq C \left( \int_{|x| \leq N} | \Delta_n(\alpha) - \Delta(\alpha) | \ud x + \sqrt{y_n} \right),
\end{align*}
where $\alpha=x+iy_n$ and the last inequality follows from Lemma \ref{lemma:nu-prop}.  So, to complete the proof we only need to estimate the integral in the last inequality above.  

We first show that for $\alpha = x + iy$, $|x| \leq N$, $|x-x_2| \geq \epsilon_1$ ($|x -x_1| \geq \epsilon_1$ if $|z| < 1$), $y \geq y_n$, $M_1 \leq |z| \leq M_2$, and all large $n$,
\begin{equation} \label{integrand-bound}
	|\Delta_n(\alpha) - \Delta(\alpha)| < \frac{C' \epsilon_0}{3} \delta_n
\end{equation}
where $\epsilon_0$ and $\epsilon_1$ come from Lemma \ref{lemma:branch-bounds} and $C'$ is a positive constant.  By Theorem \ref{thm:cubic-eq}, consider a realization where
\begin{equation*}
	|\Delta_n(\alpha) - \Delta(\alpha)| | \Delta_n(\alpha) - m_2(\alpha)| |\Delta_n(\alpha) - m_3(\alpha)| \leq C' \frac{4}{27} \epsilon_0^3 \delta_n.  
\end{equation*}
for some positive constant $C'$.  Fix $\alpha_0 = x_0 + iy_0$ with $|x_0| \leq N, y_0 > 0$, and $\min_{k=1,2} |x_0 - x_k| \geq \epsilon_1$.  Fix $z \in T$.  Choose $n$ large enough such that $|\Delta_n(\alpha_0) - \Delta(\alpha_0)| < \frac{\epsilon}{3}$.  Then for $k \in \{1,2\}$, 
\begin{equation*}
	\epsilon_0 \leq | \Delta(\alpha_0) - m_k(\alpha_0)| \leq |\Delta(\alpha_0) - \Delta_n(\alpha_0)| + |\Delta_n(\alpha_0) - m_k(\alpha_0)|
\end{equation*}
and hence 
\begin{equation*}
	\min_{k=1,2} | \Delta_n(\alpha_0) - m_k(\alpha_0)| > \frac{2 \epsilon_0}{3}.
\end{equation*}
Thus, 
\begin{equation*}
	|\Delta_n(\alpha_0) - \Delta(\alpha_0)| \leq C'\epsilon_0 \delta_n.
\end{equation*}

Next we show \eqref{integrand-bound} is true for all $y \geq y_n$, $|x| \leq N$, and $\min_{k=1,2} |x-x_k| \geq \epsilon_1$.  Suppose \eqref{integrand-bound} is false.  By continuity there exists a subsequence $n_l$, $z_l \in T$, and $\alpha_l$ with $|\Re(\alpha_l)| \leq N$ and $\Im(\alpha_l) \geq y_{n_l}$ such that 
\begin{equation*}
	|\Delta_{n_l}(\alpha_l) - \Delta(\alpha_l)| = \frac{C' \epsilon_0}{3} \delta_{n_l}.
\end{equation*}
Then 
\begin{equation*}
	|\Delta_{n_l}(\alpha_l) - \Delta(\alpha_l)| < \frac{\epsilon_0}{3}
\end{equation*}
for all $l$ greater than some $L$.  By a similar argument as above and Lemma \ref{lemma:branch-bounds}, we have
\begin{equation*}
	\min_{k=1,2}|\Delta_{n_l}(\alpha_l) - m_k(\alpha_l)| > \frac{2 \epsilon_0}{3}
\end{equation*}
for all $l > L$ and hence
\begin{equation*}
	|\Delta_{n_l}(\alpha_l) - \Delta(\alpha_l)| < \frac{C' \epsilon_0}{3} \delta_n,
\end{equation*}
a contradiction.  

Finally, for the case where $|\alpha - x_k| \leq \epsilon_1$ for $k=1$ or $2$, we apply a similar argument and Lemma \ref{lemma:branch-bounds} to obtain
\begin{equation*}
	|\Delta_n(\alpha) - \Delta(\alpha)| = O\left( \frac{\delta_n}{\sqrt{|\alpha  - x_k|}} \right) = O(\delta_n y_n^{-1/2}).
\end{equation*}

\end{proof}

\subsection{Least Singular Value Bound}

A key part of proving Lemma \ref{lemma:Y-circular} is to control the least singular value of $Y-zI$.  Equivalently, we wish to obtain control of the norm of the inverse $\|(Y-zI)^{-1}\|$.  

We will obtain a bound using the results of Tao and Vu in \cite{tao}.  We present Tao and Vu's bound on the least singular value below, which only requires a finite second moment assumption on the entries of the matrix.  

\begin{theorem}[Tao-Vu; Least singular value bound] \label{thm:tao-vu}
Let $A, C_1$ be positive constants, and let $\xi$ be a complex-valued random variable with non-zero finite variance (in particular, the second moment is finite).  Then there are positive constants $B$ and $C_2$ such that the following holds: if $N_n$ is the random matrix of order $n$ whose entries are i.i.d. copies of $\xi$, and $M$ is a deterministic matrix of order $n$ with spectral norm at most $n^{C_1}$, then,
\begin{equation} \label{tao-vu-bound}
	\Prob\left( \|(M+N_n)^{-1} \| \geq n^B \right) \leq C_2 n^{-A}.
\end{equation}
\end{theorem}

\begin{remark} \label{rem:tao-vu}
We note that the bound in \eqref{tao-vu-bound} is independent of the matrix $M$.  In particular, this bound holds for any deterministic matrix of order $n$ with spectral norm at most $n^{C_1}$.  
\end{remark}

We will prove an analogous version of Theorem \ref{thm:tao-vu} for the matrix $Y$.  We first need the following bounds for the norm of $Y$.

\begin{lemma} \label{lemma:n-norm-bound}
We have the following bounds for the norm of $Y$.
\begin{align*}
	\| Y \| = O(n) \text{ a.s.,} \\
	\E \| Y \| = O(n).
\end{align*}
We also have that for any $1\leq a \leq m$,
\begin{equation}
	\E \| X_a \| = O(n). \label{submatrix-bound}
\end{equation}
\end{lemma}

\begin{proof}
We note that 
\begin{equation*}
	Y^\ast Y = \left( \begin{array}{cccc}
                          X_m^\ast X_m &      &  &         0 \\
                              &  X_1 ^\ast X_1 &  &          \\      
                              &      &  \ddots  &         \\
                           0  &      &          & X_{m-1}^\ast X_{m-1}
                  
	\end{array}\right),
\end{equation*}
and hence the singular values of $Y$ are the the union of the singular values of $X_k$ for $1 \leq k \leq m$.  Let $s_1, \ldots s_{mn}$ denote the singular values of $Y$.  Then
\begin{align*}
	\frac{1}{mn} \| Y \| &\leq \frac{1}{mn} \sum_{j=1}^{mn} s_j \leq \frac{1}{mn}  \sum_{j=1}^{mn} s^2_j + 1 \\
		& \leq \frac{1}{mn} \tr{Y^\ast Y} +1 = \frac{1}{mn} \sum_{k=1}^m \sum_{1 \leq i, j \leq n} \left|\left(X_{k} \right)_{ij} \right|^2 + 1 \longrightarrow 2 \text{ a.s.}
\end{align*}
as $n \rightarrow \infty$ by the law of large numbers.  The same argument shows that 
\begin{equation*}
	\frac{1}{mn} \E \| Y \| \leq \frac{1}{mn} \sum_{k=1}^m \sum_{1 \leq i, j \leq n} \E \left|\left(X_{k} \right)_{ij} \right|^2 + 1 = 2.
\end{equation*}
A similar argument verifies \eqref{submatrix-bound}.

\end{proof}

\begin{theorem}[Least singular value bound for $Y$] \label{thm:singular-value}
Let $Y$ be the $(mn) \times (mn)$ matrix defined in \eqref{def-Y} and let $A$ be a positive constant.  Then, under the hypothesis of Theorem \ref{thm:main}, there exists positive constants $B$ and $C$ (depending on both $A$ and $m$) such that 
\begin{equation*}
	\Prob \left( \|(Y-zI)^{-1} \| \geq n^B \right) \leq C n^{-A}
\end{equation*}
uniformly for $|z| \leq M_2$.
\end{theorem}

\begin{proof}
We remind the reader that $(Y-zI)^{-1}$ is an $(mn) \times (mn)$ matrix and again refer to the $m^2$ blocks $(Y-zI)^{-1}_{a,b}$ each of size $n \times n$.  A simple computation reveals, that (when invertible) $(Y-zI)^{-1}_{a,b}$ has the form
\begin{equation*}
	z^\kappa X_{j_1} \cdots X_{j_l} \left(X_{i_1} \cdots X_{i_q} - z^r \right)^{-1}
\end{equation*}
where $\kappa, l, q, r$ are nonnegative integers no bigger than $m$, the variables $\kappa, l, q, r, j_1, \ldots j_l, i_1, \ldots, i_q$ depend only on $a$ and $b$, and the indices $i_1,\ldots,i_q$ are all distinct.  

By the definition of the norm, we have that
\begin{equation*}
	\| (Y-zI)^{-1} \| \leq C_m \max_{1 \leq a, b \leq m}\| (Y-zI)^{-1}_{a,b}\| \leq C_m \sum_{1 \leq a,b \leq m} \| (Y-zI)^{-1}_{a,b}\| 
\end{equation*}
where $C_m$ is a constant that depends only on $m$.  Thus, it is enough to show that given a positive constant $A$, there exists $B$ and $C$ such that 
\begin{equation*}
	\Prob \left( \| z^\kappa X_{j_1} \cdots X_{j_l} \left(X_{i_1} \cdots X_{i_q} - z^r \right)^{-1} \| \geq n^B \right) \leq C n^{-A}
\end{equation*}
uniformly for $|z| \leq M_2$.  

So we have,
\begin{align*}
	&\Prob \left( \| z^\kappa X_{j_1} \cdots X_{j_l} \left(X_{i_1} \cdots X_{i_q} - z^r \right)^{-1} \| \geq n^B \right)  \\
	& \leq  m \Prob \left( \| X_1 \| \geq n^{B/(m+2)} \right) + \Prob \left( \|\left(X_{i_1} \cdots X_{i_q} - z^r \right)^{-1}\| \geq n^{B/(m+2)} \right)
\end{align*}
for $|z| \leq M_2$ and $n$ large.  The first term can be estimated by Markov's inequality
\begin{equation*}
	\Prob \left( \| X_1 \| \geq n^{B/(m+2)} \right) \leq \frac{ \E\|X_1\| }{ n^{B/(m+2)} } = O(n^{-B/(m+2)+1})
\end{equation*}
since $\E\|X_1\| = O(n)$ by Lemma \ref{lemma:n-norm-bound}. Therefore, this term is order $n^{-A}$ by taking $B > (m+2)(A+1)$.  So, it is now enough to show that given a positive constant $A$, there exists $B$ and $C$ such that 
\begin{equation*}
	\Prob \left( \| \left(X_{i_1} \cdots X_{i_q} - z^r \right)^{-1} \| \geq n^B \right) \leq C n^{-A}.
\end{equation*} 

We note that,
\begin{equation*}
	\left(X_{i_1} \cdots X_{i_q} - z^r \right)^{-1} = X_{i_q}^{-1} \cdots X_{i_2}^{-1} \left( X_{i_1} - z^r X_{i_q}^{-1} \cdots X_{i_2}^{-1} \right)^{-1}.
\end{equation*}
By Theorem \ref{thm:tao-vu} there exists positive constants $B$ and $C$ such that
\begin{align} \label{product-bound}
	\Prob \left( \| X_{i_q}^{-1} \cdots X_{i_2}^{-1} \| \geq n^B \right) \leq m\Prob \left( \|X_1^{-1} \| \geq n^{B/m} \right) \leq C n^{-A}.
\end{align}
Thus, we only need to show that given $A$ there exists $B$ and $C$ such that
\begin{equation*}
	\Prob \left( \left( X_{i_1} - z^r X_{i_q}^{-1} \cdots X_{i_2}^{-1} \right)^{-1} \geq n^{B} \right) \leq C n^{-A}.
\end{equation*}

We then have that
\begin{align*}
	& \Prob \left( \left( X_{i_1} - z^r X_{i_q}^{-1} \cdots X_{i_2}^{-1} \right)^{-1} \geq n^{B} \right) \\
	& = \Prob \left( \left( X_{i_1} - z^r X_{i_q}^{-1} \cdots X_{i_2}^{-1} \right)^{-1} \geq n^{B} \Bigm\vert \|X_{i_q}^{-1} \cdots X_{i_2}^{-1}\| \leq n^{C_1} \right) \\
	& \qquad \qquad \times \Prob \left( \|X_{i_q}^{-1} \cdots X_{i_2}^{-1}\| \leq n^{C_1} \right)  \\
	& \qquad + \Prob \left( \left( X_{i_1} - z^r X_{i_q}^{-1} \cdots X_{i_2}^{-1} \right)^{-1} \geq n^{B} \Bigm\vert \|X_{i_q}^{-1} \cdots X_{i_2}^{-1}\| \geq n^{C_1} \right) \\
	&\qquad \qquad \times \Prob \left( \|X_{i_q}^{-1} \cdots X_{i_2}^{-1}\| \geq n^{C_1} \right) \\
	& \leq C n^{-A}
\end{align*}
where the first term is controlled by Theorem \ref{thm:tao-vu} (in particular, see Remark \ref{rem:tao-vu}) and the second term is estimated as in \eqref{product-bound}.  This completes the proof of the Theorem.  
\end{proof}

\subsection{Proof of Lemma \ref{lemma:Y-circular}}

\begin{proof}[Proof of Lemma \ref{lemma:Y-circular}]
In order to finish the proof of Lemma \ref{lemma:Y-circular} we need to show \eqref{task-reduction} holds.  By integration by parts, we have
\begin{align*}
	& \left| \int_{z \in T} (g_n(s,t) - g(s,t))e^{isu+itv} \ud t \ud s \right| \\
	& = \Bigg| - \int_{z \in T} iu \tau(s,t)\ud t \ud s + \int_{|t| \leq A^3} (\tau(A,t) - \tau(-A,t)) \ud t \\
	& \qquad -\int_{|t| \leq 1+\epsilon} \left(\tau(\sqrt{(1+\epsilon)^2 - t^2}, t) - \tau(-\sqrt{(1+\epsilon)^2-t^2}, t)\right) \ud t \\
	& \qquad +\int_{|t| \leq 1-\epsilon} \left(\tau(\sqrt{(1-\epsilon)^2 - t^2}, t) - \tau(-\sqrt{(1-\epsilon)^2-t^2}, t)\right) \ud t \Bigg|,
\end{align*}
where 
\begin{equation*}
	\tau(s,t) = e^{ius+ivt} \int_{0}^{\infty} \ln x (\nu_n(\ud x, z) - \nu(\ud x, z)).
\end{equation*}

Let $\epsilon_n = e^{-n^{\eta \delta / 16}}$.  By Theorem \ref{thm:singular-value} and the Borel-Cantelli lemma, with probability $1$, 
\begin{equation*}
	\lim_{n \rightarrow \infty} \int_{z \in T} \left| \int_0^{\epsilon_n} \ln x \nu_n(\ud x, z) \right| \ud t \ud s = 0.
\end{equation*}
By Lemma \ref{lemma:nu-prop}, 
\begin{equation*}
	\lim_{n \rightarrow \infty} \int_{z \in T} \left| \int_0^{\epsilon_n} \ln x \nu(\ud x, z) \right| \ud t \ud s = 0.
\end{equation*}
By Lemma \ref{lemma:n-norm-bound}, there exists $\kappa > 0$ such that the support of $\nu_n(\cdot, z)$ lies in $[0, \kappa n^2]$ for all $z \in T$.  Thus, by Lemma \ref{lemma:rate}
\begin{align*}
	& \int_{z \in T} \left| \int_{\epsilon_n}^\infty \ln x (\nu_n(\ud x, z) - \nu(\ud x, z)) \right| \ud t \ud s \\
	&  = \int_{z \in T} \left| \int_{\epsilon_n}^{\kappa n^2} \ln x(\nu_n(\ud x, z) - \nu(\ud x, z)) \right| \ud t \ud s \\
	&  \leq C \left[|\ln(\epsilon_n)| + \ln (\kappa n^2) \right] \max_{z \in T} \| \nu_n(\cdot, z) - \nu(\cdot, z) \| \longrightarrow 0 \qquad \text{a.s.}
\end{align*}
Therefore, with probability $1$,
\begin{equation*}
	\lim_{n \rightarrow \infty} iu\int_{z \in T} \tau(s,t) \ud s \ud t = 0.
\end{equation*}
In a similar fashion, we can show that the boundary terms satisfy the following
\begin{align*}
	\lim_{n \rightarrow \infty}& \int_{|t|\leq A^3} \tau(\pm A, t) \ud t = 0 \qquad \text{a.s.,} \\
	\lim_{n \rightarrow \infty}& \int_{|t| \leq 1+\epsilon} \tau(\pm \sqrt{(1+\epsilon)^2-t^2},t) \ud t = 0 \qquad \text{a.s.,}
\end{align*}
and
\begin{align*}
	\lim_{n \rightarrow \infty}& \int_{|t| \leq 1-\epsilon} \tau(\pm \sqrt{(1-\epsilon)^2-t^2}, t) \ud t = 0 \qquad \text{a.s.}
\end{align*}
The proof of Lemma \ref{lemma:Y-circular} is complete.
\end{proof}

\begin{remark}
After we finished our paper, we learned about a very recent preprint \cite{gotzetikh} where F.G\"{o}tze and A. Tikhomirov proved the convergence of the 
expected spectral distribution $\E \mu_X$ to the limit defined by (\ref{plotnost})  under the assumption that the matrix entries are mutually independent  
centered complex random variables with variance one.  Our approach is different from the one used in \cite{gotzetikh}.  We are grateful to Z. Burda, 
T. Tao and A. Tikhomirov for useful comments regarding the results of the paper.  In addition, we are grateful to unanimous referees for valuable and 
constructive criticism regarding the proofs of Theorem 15 and Lemma 19, and for bringing to our attention the reference \cite{girko-vlad} where 
a similar result was obtained for $m=2.$
\end{remark}

\end{document}